\documentclass[11pt]{amsart}
\usepackage{amsmath}
\usepackage{bbm}
\usepackage{mathrsfs}
\usepackage[english]{babel}
\usepackage[utf8x]{inputenc}
\usepackage[T1]{fontenc}
\usepackage{wrapfig}

\usepackage[top=3cm,bottom=2cm,left=3cm,right=3cm,marginparwidth=1.75cm]{geometry}

\usepackage{amssymb}
\usepackage{graphicx}
\usepackage[colorinlistoftodos]{todonotes}
\usepackage[colorlinks=true, allcolors=blue]{hyperref}
\usepackage{enumitem} 
\usepackage[all,cmtip]{xy}
\usepackage{tikz-cd}
\usepackage{lmodern}
\usepackage{pgfplots}
\usetikzlibrary{intersections}
\usepackage{relsize}

\pgfplotsset{compat=1.15}

\newcommand*\Laplace{\mathop{}\!\mathbin\bigtriangleup}
\newcommand{\norm}[1]{\left\lVert#1\right\rVert}

\newtheorem{theorem}{Theorem}
\newtheorem{proposition}[theorem]{Proposition}
\newtheorem{corollary}[theorem]{Corollary}

\theoremstyle{remark} 
\newtheorem{remark}[]{Remark}
\newtheorem{example}[]{Example}

\usepackage{mathtools} 
\mathtoolsset{showonlyrefs,showmanualtags}
\numberwithin{equation}{section}

\newcommand{\R}{\mathbb{R}}
\newcommand{\C}{\mathbb{C}}
\newcommand{\RP}{\mathbb{R}\mathrm{P}}

\newcommand{\cP}{\mathcal{P}}
\newcommand{\cH}{\mathcal{H}}
\newcommand{\la}{\lambda}
\renewcommand{\epsilon}{\varepsilon}
\newcommand{\vol}{\mathrm{vol}}

\usepackage[autostyle,italian=guillemets]{csquotes}

\newcommand{\be}{\begin{equation}}
\newcommand{\ee}{\end{equation}}

\title{Real lines on random cubic surfaces}

\author{Rida Ait El Manssour}
\address{MPI-MiS Leipzig, Inselstra{\ss}e 22, 04103 Leipzig, Germany}
\email{rida.manssour@mis.mpg.de}
\author{Mara Belotti}
\address{Technische Universit\"at Berlin, Chair of Discrete Mathematics/Geometry, Stra{\ss}e des 17. Juni 136, 10623 Berlin, Germany}
\email{belotti@math.tu-berlin.de}
\author{Chiara Meroni}
\address{MPI-MiS Leipzig, Inselstra{\ss}e 22, 04103 Leipzig, Germany}
\email{chiara.meroni@mis.mpg.de}

\begin{document}

\maketitle
\begin{abstract}
We give an explicit formula for the expectation of the number of real lines on a random invariant cubic surface, i.e., a surface $Z\subset \RP^3$ defined by a random gaussian polynomial whose probability distribution is invariant under the action of the orthogonal group $O(4)$ by change of variables. Such invariant distributions are completely described by one parameter $\lambda\in [0,1]$ and as a function of this parameter the expected number of real lines equals:
\be E_\la=\frac{9(8\la^2+(1-\la)^2)}{2\la^2+(1-\la)^2}\left(\frac{2\la^2}{8\la^2+(1-\la)^2}-\frac{1}{3}+\frac{2}{3}\sqrt{\frac{8\la^2+(1-\la)^2}{20\la^2+(1-\la)^2}}\right).\ee
This result generalizes previous results by S. Basu, A. Lerario, E. Lundberg and C. Peterson \cite{BLLP} for the case of a Kostlan polynomial, which corresponds to $\la=\frac{1}{3}$ and for which $E_{\frac{1}{3}}=6\sqrt{2}-3.$ Moreover, we show that the expectation of the number of real lines is maximized by random purely harmonic cubic polynomials, which corresponds to the case $\la=1$ and for which $E_1=24\sqrt{\frac{2}{5}}-3$.
\end{abstract}

\section{Introduction}

A classical result from complex algebraic geometry tells that on a \emph{generic} cubic surface in complex projective space there are exactly $27$ lines. This is still true for a generic \emph{real} cubic surface, i.e., on the zero set in complex projective space of a \emph{real} cubic polynomial, however these lines might not be real. In fact the number of real lines on the real zero locus $Z(P)\subset \RP^3$, for a generic $P\in \mathbb{R}[x_0, \ldots, x_3]_{(3)}$ in the space of real homogeneous polynomials of degree $3$ can be either $27, 15, 7$ or $3$, depending on the coefficients of the chosen polynomial \cite{Segre}.

This is a typical phenomenon in real algebraic geometry, where in general there is no ``generic'' answer to such counting problems. There is however a recent interest into looking at these questions from the probabilistic point of view, replacing the word generic with ``random'', which in the case of the current paper means asking for the expectation of the number of real lines on a random real cubic surface. This approach has its origin in classical works of Kac \cite{kac43}, Edelman and Kostlan \cite{EdelmanKostlan95}, Shub and Smale \cite{Bez2, Bez4}, and it has recently seen new progress \cite{GaWe1, GaWe3, GaWe2, LerarioFLL, NazarovSodin1, NazarovSodin2, Sarnak, SarnakWigman,Lerariolemniscate, Letwo, Lerarioshsp, LeLu:gap, Antonio, LerarioStecconi, Sarnak, PSC}, leading to the emergence of  the field of \emph{Random real algebraic geometry}. 

Of course, when talking about expected quantities, one should specify what is meant by ``random''. In this paper we will endow the space $\R[x_0, \ldots x_3]_{(3)}$ with a centered, nondegenerate gaussian distribution, which we require to be invariant under the action of the orthogonal group $O(4)$ by change of variables - so that there are no preferred points or directions in the projective space $\RP^3$. Notice that Darmois--Skitovich Theorem together with a straightforward generalization of Theorem $4.5$ of \cite{kostlan:93} guarantee that the gaussianity is a consequence of the independence of the coefficients of the monomials and the orthogonal invariance.
Such a probability distribution will be called an \emph{invariant distribution} and a polynomial sampled from it will be called an \emph{invariant polynomial}. Invariant distributions on $\R[x_0, \ldots x_3]_{(3)}$ can be explicitly described: they correspond to scalar products on $\R[x_0, \ldots x_3]_{(3)}$ which are invariant under the action of the orthogonal group $O(4)$ by change of variables, and they are parametrized by a point in the positive quadrant $(\la_1, \la_2)\in (0, \infty)\times (0, \infty)$, see \cite{kostlan:93}. This comes from the fact that there is a decomposition
\be \label{eq:deco}
\R[x_0, \ldots x_3]_{(3)}=\cH_3\oplus \|x\|^2\cdot\cH_1,
\ee
where $\cH_3$ and $\cH_1$ denotes respectively the space of harmonic cubic polynomials and harmonic linear polynomials (i.e., just linear polynomials). The remarkable fact here is that the decomposition \eqref{eq:deco} is orthogonal with respect to any invariant scalar product; moreover the action of the orthogonal group by change of variables preserves the two spaces of harmonics and in addition the induced representation on these spaces is irreducible. In particular, in each space of harmonics, there is a unique (up to multiples) scalar product which is $O(4)$--invariant; this explains the two positive parameters needed to describe an invariant distribution.

In practice, in order to construct a random invariant polynomial, we proceed as follows. First observe that the quantity we are interested in (the number of lines on the zero set, and in fact the zero set itself) does not depend on the multiple of the defining polynomial that we take and we can normalize our parameters to satisfy $\la_1+\la_2=1$. In particular we can work with a single parameter $\la\in (0,1)$ such that $(\la_1, \la_2)=(\la, 1-\la)$. Consider the $L^2$--scalar product, which is defined for $f, g\in \R[x_0, \ldots x_3]_{(3)}$ by
\be 
( f, g )_{L^2}=\frac{1}{2\pi^2} \int_{\R^4}f(x)\,g(x) e^{-\frac{\norm{x}^2}{2}} \,dx.
\ee
Then we fix bases $\{H_{3, j}\}_{j\in J_3=\{1,\dots ,16\}}$ for $\cH_3$ and $\{H_{1, j}\}_{j\in J_1=\{1,\dots ,4\}}$ for $\cH_1$ which are orthonormal with respect to the $L^2$--scalar product.
With these choices we define a random polynomial $P_\la$ as a linear combination of random harmonics, weighted by the parameters:
\be\label{eq:invariant} P_{\la}(x)=\la\left(\sum_{j\in J_3}\xi_{3, j}\cdot H_{3, j}(x)\right)+(1-\la)\left(\sum_{j\in J_1}\xi_{1, j}\cdot \|x\|^2H_{1, j}(x)\right),\ee
where $\{\xi_{3, j}\}_{j\in J_3}$ and $\{\xi_{1, j}\}_{j\in J_1}$ are two independent families of independent standard gaussians. We include in our study also the choices $\la=0$ and $\la=1$, which correspond to purely harmonic polynomials (but not to scalar products). The case $\lambda = 0$ is special also for another reason: the associated hypersurface is a degenerate cubic, namely a hyperplane.

\begin{example}[The Kostlan distribution]A Kostlan random polynomial is defined by
\be\label{eq:kostlan} P(x)=\sum_{|\alpha|=3}\xi_\alpha \cdot \left(\frac{3!}{\alpha_0!\cdots\alpha_3!}\right)^{1/2}x_0^{\alpha_0}\cdots x_3^{\alpha_3}\ee
where $\{\xi_{\alpha}\}_{|\alpha|=3}$ is a family of independent standard gaussians. The resulting probability distribution  on $\R[x_0, \ldots x_3]_{(3)}$ is invariant and corresponds to the choice of $\lambda=\frac{1}{3}$ in \eqref{eq:invariant} (see Corollary \ref{cor:kostlan}). The authors of \cite{BLLP} have proved that the expectation of the number of real lines on the zero set of a random Kostlan cubic equals:
\be \label{eq:lineskostlan} E_{\frac{1}{3}}=6\sqrt{2}-3.\ee
\end{example}
Generalizing the work of \cite{BLLP}, in this paper we give an explicit formula for the expectation of the number of real lines on a random invariant cubic, as a function of the parameter $\lambda\in [0,1]$.
\begin{theorem}\label{thm:main}The expected number of real lines on the zero set of the random cubic polynomial $P_\la$ equals:
\be \label{eq:Elambda}
E_\la=\frac{9(8\la^2+(1-\la)^2)}{2\la^2+(1-\la)^2}\left(\frac{2\la^2}{8\la^2+(1-\la)^2}-\frac{1}{3}+\frac{2}{3}\sqrt{\frac{8\la^2+(1-\la)^2}{20\la^2+(1-\la)^2}}\right).
\ee
\end{theorem}
An interesting corollary of the previous Theorem is the fact that one can analytically prove that the expectation is maximized at $\lambda=1$, i.e., for random  purely harmonic cubics.
\begin{corollary}\label{cor_maximum}
The function $E_\lambda$ is monotone increasing and attains its maximum at $\lambda=1$:
\be 
E_1=24\sqrt{\frac{2}{5}}-3.
\ee
\end{corollary}
\begin{remark}The previous corollary is particularly interesting because it confirms the intuition that purely harmonic polynomials of maximum degree exhibit complicated topological configurations, see \cite{Kozhasov}.
\end{remark}

\begin{remark}
On the other hand the minimum of the function $E_{\lambda}$ is $3$, and this number has a precise meaning. In fact we prove in Proposition \ref{prop_perturbH1} that there exists a neighborhood of the set of purely harmonic polynomials of degree one (i.e. linear form times $\|x\|^2$), such that the smooth cubic surfaces in this neighborhood contain exactly three lines. The proof of this fact does not involve the expression of the function $E_{\lambda}$, therefore we could deduce that $E_{0} = 3$ without knowing \eqref{eq:Elambda}.
\end{remark}

\begin{remark}Another possible model of random cubics can be introduced following the work of Allcock, Carlson, and Toledo \cite{ACD}. They have studied the moduli space of real cubic surfaces from the point of view of hyperbolic geometry and computed the orbifold Euler characteristic (which is proportional to the hyperbolic volume) of each component of the moduli space. One can define an expectation taking the weighted average of the number of real lines, weighted by the volume of the corresponding component. In this way one gets an expected number of $\frac{239}{37}$ real lines, see \cite[Table 1.2]{ACD}. 
\end{remark}

\begin{remark}Yet another model of randomness can be obtained by looking at random \emph{determinantal} cubics. To be more specific, consider random $3\times 3$ matrices $A_0, A_1, A_2, A_3$ filled with independent standard gaussians, and define the random polynomial:
\be \label{eq:determinantal}F(x_0, x_1, x_2, x_3)=\det(x_0A_0+x_1A_1+x_2A_2+x_3A_3).\ee
Random determinantal cubics are $O(4)$--invariant. Smooth cubics admit a determinantal representation (i.e., they can be written as the zero set of some $F$ as in \eqref{eq:determinantal}), see \cite{Beauville, Buckley}. It is natural therefore to ask for the expectation of the number of real lines on a random determinantal cubic surface $Z(F)\subset \RP^3$, however this problem seems to be considerably more complicated than the gaussian one considered here (the coefficients of $F$ in \eqref{eq:determinantal} are cubic in gaussian variables and they are also highly dependent) and we leave this as an open question. 

\end{remark}

\begin{remark}As said in the beginning, our approach will be probabilistic and our answer will depend on the probability distribution we have chosen. It is important to mention that there exists also a certain signed count of lines on a (generic) real cubic surface that is independent of the surface itself. For this type of count, following classical work of Segre \cite{Segre} (later rediscovered and extended by Okonek and Teleman \cite{OkTel} and Kharlamov and Finashin \cite{finkh}), one can classify the lines lying on the cubic into elliptic and hyperbolic. This corresponds to giving a sign to each line. The number $e$ of elliptic lines plus the number $h$ of hyperbolic lines depend on the cubic, but their difference $h-e$ is always $3$. 
Following \cite{KassWickelgren}, one can further extend this type of signed count to a different field $\mathbb{K}$ (for instance the $p$--adic numbers $\mathbb{K}=\mathbb{Q}_p$). In this case a line is a closed point in the Grassmannian of lines in $\mathbb{P}_\mathbb{K}^3$. The sign, which is now called \emph{type}, takes value in the Grothendieck–Witt group GW($\mathbb{K}$) of non--degenerate bilinear forms and it depends on the field of definition of the line.
With these specifications we get a similar invariant count, see \cite[Theorem 2]{KassWickelgren}. An interesting question is: what happens over the $p$--adic numbers? In this direction, \cite[Theorem 2]{KassWickelgren} gives a way to perform a well defined enriched count but, in the spirit of the current paper, it makes sense to ask for the expected number of $\mathbb{Q}_p$--lines on a random $p$--adic cubic. This question has been studied by the first named author of this paper together with Lerario in \cite{manssour2020probabilistic} and the answer is $\tfrac{(p^3-1)(p^2+1)}{p^5-1}$.
\end{remark}

\begin{remark}
The study of cubic surfaces  has been recently enriched by the famous $27$ questions posed by Sturmfels, that are collected in \cite{S27Q}. By the same logic of this paper, it can be noticed that some of those questions can be restated according to a probabilistic point of view. For example looking at question $23$ and putting a probability distribution on $\RP^{19}$, instead of asking for a semialgebraic description of the set of smooth hyperbolic cubics in $\RP^{19}$, one could seek the probability of a smooth cubic to be hyperbolic.
\end{remark}

\subsection*{Acknowledgements}
We, the authors, wish to thank Antonio Lerario, our professor and the person without whom this article would not exist. Thank you for believing in us, in us who "attiriamo altre basi reali su armoniche random" ($=$we attract other real bases on random harmonics), and for your amazing ability in creating anagrams. Special thanks also to SISSA, the place where this article was born, and to the Max Planck Institute for Mathematics in the Sciences, Leipzig, for making our first little research experience great. We would also like to thank the ICERM for the beautiful workshop on Symmetry, Randomness, and Computations in Real Algebraic Geometry, that was the occasion for interesting discussions.

\section{Preliminaries}
\subsection{The decomposition into harmonic polynomials and invariant scalar products}\label{harmonics}

Let us consider the space of real $d$--homogeneous polynomials $W_{n,d}=\R[x_0, \ldots x_n]_{(d)}$. The orthogonal group $O(n+1)$ acts on it by change of variables, so that we can view $W_{n,d}$ as a representation of $O(n+1)$. We want to find the decomposition of  $W_{n,d}$ into its irreducible subrepresentations.
Denote the space of real homogeneous harmonic polynomials of degree $d$ in $n + 1$ variables by
\[
\cH^{n}_{d}:=\{H \in W_{n,d} : \Laplace H = 0\}.
\]
This space is invariant with respect to $O(n+1)$ and the following algebraic decomposition holds (see \cite{BLLP}):
\begin{equation}\label{eq:decirr}
W_{n,d}=\bigoplus_{d-j\in2\mathbb{N}}\norm{x}^{d-j}\cH_{j}^{n}.
\end{equation}
Moreover the spaces $\norm{x}^{d-j} \cH^{n}_{j}$ form irreducible representations of $O(n+1)$ and are orthogonal with respect to any $O(n+1)$--invariant scalar product.
Let us denote with $(\cdot ,\cdot)$ a generic real scalar product on $W_{n,d}$ which is invariant under the action of the orthogonal group $O(n+1)$; we will use the notation $(\cdot ,\cdot)_2$ for the $L^2$ scalar product which is by definition
\begin{equation}
(f,g)_2 = \frac{1}{2\pi^2}\int_{\R^{n+1}} f(x)g(x) e^{-\frac{\norm{x}^2}{2}} dx \qquad f,g\in W_{n,d}.
\end{equation}
As a consequence of Schur Lemma (see \cite[Lemma 18.1.1]{tao:2014}) the restriction of $(\cdot ,\cdot)$ to the space $\norm{x}^{d-j}\cH_{j}^{n}$ is a multiple of the $L^2$ scalar product. So given $f,g\in W_{n,d}$ we can write $f=\sum_{j}\norm{x}^{d-j}f_j$ and $g=\sum_{j}\norm{x}^{d-j}g_j$, with $f_j,g_j\in \cH_j^n$ where $j$ is such that $d-j\in 2\mathbb{N}$, and we have that
\[
(f,g)=\sum_{d-j\in2\mathbb{N}}\mu_j(\norm{x}^{d-j}f_j,\norm{x}^{d-j}g_j)_2
\]
for some $\mu_d$, $\mu_{d-2}$, $\ldots > 0$.

Given an invariant scalar product $(\cdot,\cdot)$ we can construct a gaussian probability distribution which is invariant under rotations. First we fix an orthonormal basis $\{\norm{x}^{d-j} H_{j,i}\}_{i \in J_j}$ for harmonics $\mathcal{H}^n_j$ with respect to $(\cdot,\cdot)_2$, where $J_j = \{1, \ldots , \hbox{dim}(\cH^n_j)\}$. Then $\{\lambda_j \norm{x}^{d-j} H_{j,i}\}$ is an orthonormal basis with respect to $(\cdot,\cdot)$, where $\lambda_j = \mu_j^{-\frac{1}{2}}$. We construct a random polynomial with such a basis whose coefficients are given by centered gaussian random variables $\xi_{j,i}\sim N(0,1)$:
\begin{equation}
P(x) = \sum_{d-j \in 2\mathbb{N}} \la_j \sum_{i\in J_j} \xi_{j,i} \norm{x}^{d-j} H_{j,i}(x).
\end{equation}
In our case we have that 
\[
W_{3,3} =\cH_3^3\oplus\norm{x}^2 \cH_1^3
\]
and therefore we only need two parameters to classify all the scalar products
\[
(\cdot,\cdot)=\mu_1(\cdot,\cdot)_2+\mu_2(\cdot,\cdot)_2.
\]
\begin{table}[t]
\caption{Orthogonal basis for Theorem \ref{thm:main}}
\begin{center}
\begin{tabular}{ |c| c| }
\hline
\rule[-2mm]{0mm}{0.65cm}
Basis for $\cH_3^3$ & $\norm{\cdot}_{L^2}$ \\[0.6ex]
\hline
\rule{0mm}{0.4cm}
$x_0x_1x_2$ & 1  \\[0.6ex]
$x_0x_1x_3$ & 1  \\[0.6ex] 
$x_0x_2x_3$ & 1\\[0.6ex]    
$x_1x_2x_3$ & 1\\[0.6ex]
$x_0^3-x_0(x_1^2+x_2^2+x_3^2)$ & $2\sqrt{3}$\\[0.6ex] 
$x_1^3-x_1(x_0^2+x_2^2+x_3^2)$ & $2\sqrt{3}$\\[0.6ex]
$x_2^3-x_2(x_0^2+x_1^2+x_3^2)$ & $2\sqrt{3}$\\[0.6ex] 
$x_3^3-x_3(x_0^2+x_1^2+x_2^2)$ & $2\sqrt{3}$\\[0.6ex] 
$x_0(x_1^2-x_2^2)$ & 2\\[0.6ex] 
$x_1(x_2^2-x_3^2)$ & 2\\[0.6ex] 
$x_2(x_1^2-x_0^2)$ & 2\\[0.6ex]  
$x_3(x_0^2-x_1^2)$ & 2\\[0.6ex]  
$x_0(x_1^2-x_3^2)-\frac{1}{2}x_0(x_1^2-x_2^2)$ & $\sqrt{3}$\\[0.6ex] 
$x_1(x_2^2-x_0^2)-\frac{1}{2}x_1(x_2^2-x_3^2)$ & $\sqrt{3}$\\[0.6ex] 
$x_2(x_1^2-x_3^2)-\frac{1}{2}x_2(x_1^2-x_0^2)$ & $\sqrt{3}$\\[0.6ex] 
$x_3(x_0^2-x_2^2)-\frac{1}{2}x_3(x_0^2-x_1^2)$ & $\sqrt{3}$\\[0.6ex]
\hline

\rule[-2mm]{0mm}{0.65cm}
Basis for $\|x\|^2\cH_1^3$ & $\norm{\cdot}_{L^2}$\\[0.6ex] 
\hline
\rule{0mm}{0.4cm}
 $x_0(x_0^2+x_1^2+x_2^2+x_3^2)$ & $4\sqrt{3}$  \\[0.6ex]
 $x_1(x_0^2+x_1^2+x_2^2+x_3^2)$  & $4\sqrt{3}$\\[0.6ex] 
$x_2(x_0^2+x_1^2+x_2^2+x_3^2)$ & $4\sqrt{3}$\\[0.6ex] 
 $x_3(x_0^2+x_1^2+x_2^2+x_3^2)$ & $4\sqrt{3}$\\[0.6ex]
\hline 
\end{tabular}\label{table}
\end{center}
\end{table}
Let us fix bases $\{H_{3, i}\}_{i\in J_3=\{1,\dots ,16\}}$ for $\cH_3^3$ and $\{\|x\|^2 H_{1, i}\}_{i\in J_1=\{1,\dots ,4\}}$ for $\|x\|^2 \cH_1^3$ which are orthonormal with respect to the $L^2$-scalar product, and then we have that $\{\frac{1}{\sqrt{\mu_1}}H_{3, i}\}_{i\in J_3}\cup\{\frac{\|x\|^2}{\sqrt{\mu_2}}H_{1, i}\}_{i\in J_1}$ is an orthonormal basis with respect to our scalar product. Notice that since for our purposes we just need to classify scalar products up to constants, we can rescale our parameters such that they sum up to $1$ and obtain the following random polynomial
\be \label{eq:Plambda}
P_{\la}(x)=\la\left(\sum_{i\in J_3}\xi_{3, i}\cdot H_{3, i}(x)\right)+(1-\la)\left(\sum_{i\in J_1}\xi_{1, i}\cdot \|x\|^2H_{1, i}(x)\right)
\ee
where $\xi_{j,i}\sim N(0,1)$ are independent standard gaussians. In Theorem \ref{thm:main} we will use the explicit orthogonal basis for $\cP$ shown in table \ref{table}.

\begin{remark}
Notice that we will take into account also the limit cases $\la = 0$ and $\la = 1$ of pure harmonics of degree $1$ and $3$ respectively.
\end{remark}

\begin{corollary}\label{cor:kostlan}
The Kostlan distribution \eqref{eq:kostlan} corresponds to the choice $\lambda=\frac{1}{3}$ in \eqref{eq:invariant}.
\end{corollary}
\begin{proof}
Take the element $x_0x_1x_2$ of the basis. Its $L^2$ norm is 1, while its Kostlan norm is $\frac{1}{\sqrt{6}}$, therefore we get that $\mu_1=\frac{1}{6}$.
Consider now the element $x_0(x_0^2+x_1^2+x_2^2+x_3^2)$. Its $L^2$ norm is $4\sqrt{3}$, while its Kostlan norm is $\sqrt{2}$, therefore $\mu_2=\frac{1}{24}$. We look for $\alpha\in\R$ such that $\alpha\frac{1}{\sqrt{\mu_1}}+\alpha\frac{1}{\sqrt{\mu_2}}=1$, i.e., $\alpha=\frac{1}{3\sqrt{6}}$. So in the end $\lambda=\alpha\frac{1}{\sqrt{\mu_1}}=\frac{1}{3}$.
\end{proof}

\subsection{Vector bundles and the Kac--Rice formula}

In this section we recall the construction from \cite[Theorem $1$]{BLLP}.
Let $Gr^+(2,4)$ denote the Grassmannian of oriented $2$--planes in $\R^4$, that we identify with its image in $\mathbb{S}^5$ under the spherical Pl\"ucker embedding. It can be seen as the set of simple, norm--one vectors in the second exterior power of $\R^4$. Denote by $g$ the Riemannian metric induced by this embedding.  Let $\hbox{Sym}^3(\tau^*_{2,4})$ be the $3^{rd}$ symmetric power of the dual of the tautological bundle on $Gr^+(2,4)$. For every $f\in \R[x_0,\ldots ,x_3]_{(3)}$, we define a section $\sigma_f$ of the bundle $\hbox{Sym}^3(\tau^*_{2,4})$ by considering $\sigma_f(W)=f|_W$, its restriction on $W\in Gr^+(2,4)$. In this way our main problem of finding the expected number of lines in the surface $Z(P_{\la})\subseteq \RP^3$ becomes computing
$$E_\lambda =\mathbb{E}\#\lbrace W\in Gr(2,4) \mid \sigma_{P_\lambda}(W)=0\rbrace$$
where $Gr(2,4)$ denotes the Grassmannian of $2$--planes in $\R^4$, whose double cover is given by $Gr^+(2,4)$.
We recall now the following theorem which is an essential tool for this computation.

\begin{theorem}[Kac--Rice formula \cite{adlerrandom}]
Let $(M,g)$ be a Riemannian manifold of dimension $m$ and $X:M\to \mathbb{R}^m$ be a smooth random map such that 
\begin{enumerate}[label=(\roman*)]
\item for every $t\in M$, the random vector $X(t)$ has a gaussian nondegenerate distribution;
\item the probability that $X$ has degenerate zeroes in $M$ is zero.
\end{enumerate}
Then, denoting by $p_{X(t)}$ the density function of $X(t)$, for every $U\subset M$ measurable set the expected number of zeroes of $X$ in $U$ is given by the formula:
\begin{equation}
\mathbb{E}\,\#(\{X=0\}\cap U)=\int_{U} \mathbb{E}\lbrace |{\det(\hat{J}X(t))}| \mid X(t)=0 \rbrace \, p_{X(t)}(0)\cdot w_U(t)
\end{equation}
where $w_U$ is the volume form induced by the Riemannian metric $g$ and $\hat{J}X(t)$ denotes the matrix of the derivatives of the components of $X$ with respect to an orthonormal frame field.
\end{theorem}

For $i=1,2$ and $j=3,4$, consider $E_{i,j}$ the matrix that has $1$ in position $(i,j)$, $-1$ in position $(j,i)$ and $0$ otherwise; then $e^{tE_{i,j}}\in O(4)$. Let $e_0,e_1,e_2,e_3$ be the standard basis vectors of $\R^4$, $t = (t_{1,3},t_{1,4},t_{2,3},t_{2,4})\in \R^{2 \times 2}$, and consider the function

\[
\xymatrix@R=0pt@C=1pc{
R \colon \mathbb{R}^{2 \times 2} \times \hbox{span}\{e_0,e_1\} \ar[r] & \mathbb{R}^4\\
{\hphantom{R\colon{}}} (t,y) \ar@{|->}[r] & (e^{\sum t_{i,j}E_{i,j}})\cdot y
}
.\]

Then $\phi \colon \mathbb{R}^{2\times 2} \to Gr^+(2,4)$ defined by $\phi (t)=R(t,e_0)\wedge R(t,e_1)$ is a local parametrization of $Gr^+(2,4)$ around $e_0\wedge e_1$. In fact this is the Riemannian exponential map centred at $e_0\wedge e_1$ (see \cite{Kozlov}).

Hence $\phi^{-1} \colon U \to \mathbb{R}^{2\times 2}$ is a coordinate chart on a neighborhood $U$ of $e_0\wedge e_1$, and we get a trivialization of the bundle $\hbox{Sym}^3(\tau ^*_{2,4})$ over $U$ as follows:
\[
\begin{tikzcd}
\hbox{Sym}^3(\tau^*_{2,4})|_U \arrow[dr, "\pi"] \arrow[rr, "h"] & & U\times \mathbb{R}[y_0,y_1]_{(3)} \arrow[ld, "p_1"]\\
&U&
\end{tikzcd} 
\]
where $h(f)=(W,f(R(\phi^{-1}(W),\cdot)))$ for every $f\in \hbox{Sym}^3(\tau^*_{2,4})|_W$.

\begin{remark}
Since $Gr^+(2,4)$ is compact and connected, and the map $\phi$ is a Riemannian exponential map, then $\phi$ is surjective. We can take $U$ to be the largest domain for which $\phi$ is a diffeomorphism: then $Gr^+(2,4)\setminus U$ is the cut locus at $e_0\wedge e_1$ (see \cite[Theorem III.2.2]{chavel2006riemannian}) and it has measure $0$. So integrating over $U$ is equivalent to integrating over $Gr^+(2,4)$.
\end{remark}
 
Take the polynomial $P_\lambda$ in \eqref{eq:Plambda} and define
\[
\Tilde{\sigma}_{P_\lambda} \colon U \to \mathbb{R}[y_0,y_1]_{(3)} \simeq \R^4
\]
in such a way that $h(\sigma_{P_\lambda}(W))=(W,\Tilde{ \sigma}_{P_\lambda}(W))$.
So we can apply the Kac--Rice formula to $X=\Tilde{\sigma}_{P_\lambda}\circ \phi \colon \phi^{-1}(U) \to \R^4$:
 \begin{align}
 \mathbb{E}\#\{\Tilde{\sigma}_{P_\lambda}=0\}=\mathbb{E}\#\{X=0\}&=\int_{\phi^{-1}(U)} \mathbb{E}\{|\det(\hat{J}X(t))| \mid X(t)=0\}p_{X(t)}(0)\cdot \phi^*w_{Gr^+(2,4)}(t)\\
&= \int_U \mathbb{E}\{|\det(J(W))| \mid \Tilde{\sigma}_{P_\lambda}(W)=0\}p(0,W)\cdot w_{Gr^+(2,4)}(W)\\
&= \int_{Gr^+(2,4)} \mathbb{E}\{|\det(J(W))| \mid \Tilde{\sigma}_{P_\lambda}(W)=0\}p(0,W)\cdot w_{Gr^+(2,4)}(W),
\end{align}
 where here $\phi^{-1}(U)$ is endowed with the pull--back metric $\phi^*g$, $p(0,W)$ denotes the density at zero of $\Tilde{\sigma}_{P_\lambda}(W)$ and $J(W)$ is the matrix of the derivatives at $W$ of the components of $\Tilde{\sigma}_{P_\lambda}$ with respect to an orthonormal frame field, that we will simply call Jacobian matrix.
 
The fact that the distribution of $P_\lambda$ is $O(4)$--invariant implies that the function 
\[
C(W)\coloneqq \mathbb{E}\{|\det(J(W))| \mid \Tilde{\sigma}_{P_\lambda}(W)=0\}p(0,W)
\]
is a constant $C$ which does not depend on $W$. Indeed, let $W_1$ and $W_2$ be two elements of $Gr^+(2,4)$, and let $k\in O(4)$ be such that $k(W_1)=W_2$. Then, by the Kac--Rice formula, we have 
\begin{align*}
    C(W_1)&=\lim_{\epsilon\rightarrow+\infty}\frac{1}{\vol(B(W_1,\epsilon))}\int_{B(W_1,\epsilon)}\mathbb{E}\{|\det(J(W))| \mid \Tilde{\sigma}_{P_\lambda}(W)=0\}p(0,W)\cdot w_{Gr^+(2,4)}(W)\\
    &=\lim_{\epsilon\rightarrow+\infty}\frac{1}{\vol(B(W_1,\epsilon))}\mathbb{E}\#(\{\Tilde{\sigma}_{P_\lambda}(W)=0\}\cap B(W_1,\epsilon))\\
    &=\lim_{\epsilon\rightarrow+\infty}\frac{1}{\vol(B(W_1,\epsilon))}\mathbb{E}\#(\{\Tilde{\sigma}_{P_\lambda}\circ k^{-1}(W)=0\}\cap k(B(W_1,\epsilon)))\\
    &=\lim_{\epsilon\rightarrow+\infty}\frac{1}{\vol(B(W_2,\epsilon))}\mathbb{E}\#(\{\Tilde{\sigma}_{P_\lambda}(W)=0\}\cap B(W_2,\epsilon))\\
    &=C(W_2)
\end{align*}
denoting by $B(W_i,\epsilon)$ the ball around $W_i$ of radius $\epsilon$. Therefore the expected number of zeros of the section is
\[
\mathbb{E}\#\{\Tilde{\sigma}_{P_\lambda}=0\}=C\cdot \vol(Gr^+(2,4))
\]
where $\vol(Gr^+(2,4))$ is the volume of $Gr^+(2,4)$.
Moreover we will show in the proof of Theorem \ref{thm:main} that $\Tilde{\sigma}_{P_\lambda}(W)$ and $J(W)$ are independent random variables (for a certain $W$), and in that case
\[
\mathbb{E}\{|\det(J(W))| \mid \Tilde{\sigma}_{P_\lambda}(W)=0\}= \mathbb{E}\{|\det(J(W))|\}.
\]
Because $Gr^+(2,4)$ is a double covering of $Gr(2,4)$, in the end we get that
\begin{align}
E_\lambda &=\mathbb{E}\#\{ W\in Gr(2,4) \mid \sigma_{P_\lambda}(W)=0\} \notag \\
&=\frac{1}{2}\,\mathbb{E}\#\{ W\in Gr^+(2,4) \mid \sigma_{P_\lambda}(W)=0\} \notag \\
&=\frac{1}{2}\,C\cdot \vol(Gr^+(2,4)) \notag \\
&=\mathbb{E}\{|\det(J(W_0))|\}\cdot \vol(Gr(2,4))\cdot p(0,W_0)\label{Elambda-KR}
\end{align}
for a fixed $W_0\in Gr^+(2,4)$.
 
Let us now focus on the Jacobian matrix: write the polynomial $P_\lambda$ in the monomial basis as 
\[
P_\lambda= \sum_{|i|=3} \beta_{i_0,i_1,i_2,i_3}y_0^{i_0}y_1^{i_1}y_2^{i_2}y_3^{i_3}
\]
and choose $W_0=e_0\wedge e_1$; since $W_0=\phi(0)$ then
\[
\Tilde{\sigma}_{P_\lambda}(W_0)=\sigma_{P_\lambda}(W_0)=\sum_{|i|=3} \beta_{i_0,i_1,0,0}y_0^{i_0}y_1^{i_1}.
\]
As in the proof of \cite[Theorem 2]{BLLP} we can compute the matrix $J(W_0)$ that turns out to be: 
\[J(W_0)=\begin{bmatrix}
\beta_{2,0,1,0} & 0 & \beta_{2,0,0,1} & 0 \\
\beta_{1,1,1,0} & \beta_{2,0,1,0} & \beta_{1,1,0,1} & \beta_{2,0,0,1}\\
\beta_{0,2,1,0} & \beta_{1,1,1,0} & \beta_{0,2,0,1} & \beta_{1,1,0,1} \\
0 & \beta_{0,2,1,0} & 0 & \beta_{0,2,0,1}
\end{bmatrix}
\]
This matrix will be used in the proof of the main theorem.
\section{Proof of Theorem \ref{thm:main}}

\begin{proof}

Fix the orthogonal basis $\{\Tilde{H}_{3,j}\}_{j\in\{1,\dots ,16\}} \cup \{\Tilde{H}_{1,j}\}_{j\in\{1,\dots ,4\}}$ introduced in Table \ref{table} for the space $W_{3,3}$. Then our random polynomial is
\begin{multline}
P_{\la}(x)=\la\left(\sum_{j=1}^4\xi_{3, j}\cdot \Tilde{H}_{3, j}(x) + \sum_{j=5}^8\xi_{3, j}\cdot \frac{\Tilde{H}_{3, j}(x)}{2\sqrt{3}} + \sum_{j=9}^{12}\xi_{3, j}\cdot \frac{\Tilde{H}_{3, j}(x)}{2} + \sum_{j=13}^{16}\xi_{3, j}\cdot \frac{\Tilde{H}_{3, j}(x)}{\sqrt{3}} \right)\\
+(1-\la)\left(\sum_{j=1}^4\xi_{1, j}\cdot \|x\|^2 \frac{\Tilde{H}_{1, j}(x)}{4\sqrt{3}}\right).
\end{multline}
Expanding this harmonic basis in the monomial one we can compute directly the Jacobian as above and we obtain the expression:

\[J(W_0)=\begin{bmatrix}
\bar{x}- \bar{y} & 0 & \bar{x'} - \bar{y'} & 0 \\
\bar{z} & \bar{x}- \bar{y} & \bar{z'} & \bar{x'} - \bar{y'}\\
\bar{x} + \bar{y} & \bar{z} & \bar{x'} + \bar{y'} & \bar{z'} \\
0 & \bar{x} + \bar{y} & 0 & \bar{x'} + \bar{y'}
\end{bmatrix}
\]
where these new gaussians
\begin{align*}
&\bar{x} = -\frac{\la}{2\sqrt{3}}\xi_{3,7} + \frac{\la}{2\sqrt{3}}\xi_{3,15} + \frac{(1-\la)}{4\sqrt{3}}\xi_{1,3} &\quad &\sim N\left(0,\sqrt{\frac{\la^2}{6} + \frac{(1-\la)^2}{48}}\right)\\
&\bar{x'} = -\frac{\la}{2\sqrt{3}}\xi_{3,8} + \frac{\la}{2\sqrt{3}}\xi_{3,16} + \frac{(1-\la)}{4\sqrt{3}}\xi_{1,4} &\quad &\sim N\left(0,\sqrt{\frac{\la^2}{6} + \frac{(1-\la)^2}{48}}\right)\\
&\bar{y} = \frac{\la}{2}\xi_{3,11} &\quad &\sim N\left( 0, \frac{\la}{2} \right)\\
&\bar{y'} =- \frac{\la}{2}\xi_{3,12} &\quad &\sim N\left( 0, \frac{\la}{2} \right)\\
&\bar{z} = \la \, \xi_{3,1} &\quad &\sim N\left( 0, \la \right)\\
&\bar{z'} = \la \,\xi_{3,2} &\quad &\sim N\left( 0, \la \right)
\end{align*} 
are again independent.
On the other hand, when we compute $\Tilde{\sigma}_{P_\lambda}(W_0)$ the only basis elements that do not vanish are $H_{3,5}, H_{3,6}, H_{3,9}, H_{3,13}, H_{3,14}, H_{1,1}, H_{1,2}$ so this section and $J(W_0)$ are independent.
Therefore, thanks to equation \eqref{Elambda-KR}, we are left with
\[
E_\la = \mathbb{E}\{|\det(J(W_0))|\}\cdot \vol(Gr(2,4))\cdot p(0,W_0).
\]

Let us compute $\mathbb{E}\{|\det(J(W_0))|\}$. We will use the following notation: if $\bar{t} \sim N(0,\eta )$ we will call $t= \frac{1}{\eta} \bar{t} \sim N(0,1)$. It turns out after some computations that
$$
\det(J(W_0)) = \left( \frac{\la^2}{6} + \frac{(1-\la)^2}{48} \right)\la^2 \alpha^2 - \frac{\la^4}{4} \beta^2 + \left( \frac{\la^2}{6} + \frac{(1-\la)^2}{48} \right)\la^2 \gamma^2
$$
where $\alpha=xy'-x'y$, $\beta=y'z - yz'$, $\gamma=x'z-xz'$ are quadratic forms in gaussians.

Instead of parametrizing the scalar products with $(\la,1-\la)$ we can use other rescaled parameters $(M,N)$ such that $\frac{M^2}{6}+\frac{N^2}{48}=1$. Fix $(\la,1-\la)$ and $(M,N)$ parametrizing the same distribution: there exists $\mu:[0,1]\to{(0,\infty)}$ such that $\mu(\lambda) P_{\la}(x) = \check{P}_{M,N}$ as explained in section \ref{harmonics}, where
\[
\check{P}_{M,N} = M \left(\sum_{j\in J_3}\xi_{3, j}\cdot H_{3, j}(x)\right)+N \left(\sum_{j\in J_1}\xi_{1, j}\cdot \|x\|^2H_{1, j}(x)\right)
\]
and $\mu(\lambda) = \frac{4\sqrt{3}}{\sqrt{(1-\la)^2 + 8\la^2}}$.
Hence we can do again the same reasoning using the $(M,N)$ parameters, and we can compute for this polynomial the function
\[
\check{E}_{M,N} = \mathbb{E}\{|\det(J(W_0))|\}\cdot \vol(Gr(2,4))\cdot p(0,W_0).
\]
To write then the expectation as a function of the $\lambda$ parameter we need to remember that $E_{\la} = \check{E}_{\la\,\mu(\lambda), (1-\la)\mu(\lambda)}$, as the zero set does not change under multiplication of a polynomial by a constant.

With these new parameters the determinant becomes much simpler:
$$
\det(J(W_0)) = M^2\left(\alpha^2 - \frac{M^2}{4}\beta^2 + \gamma^2\right).
$$
In order to compute the expectation of $|\det(J(W_0))|$ we need the joint density $\rho(\alpha,\beta,\gamma)$. Surprisingly it can be recovered by the method of characteristic functions and using Theorem $2.1$ of \cite{Scar}, as explained in \cite{BLLP}, so that denoting by $| \cdot |$ the Euclidean norm:
$$
\rho(\alpha,\beta,\gamma) = \frac{1}{4\pi} \frac{e^{-|(\alpha,\beta,\gamma)|}}{|(\alpha,\beta,\gamma)|}.
$$
Therefore we can compute the expectation of $|\det(J(W_0))|$ as:
\begin{align}
\mathbb{E}\{|\det(J(W_0))|\} &= \frac{M^2}{4\pi} \int_{\R^3} \left| \alpha^2 - \frac{M^2}{4}\beta^2 + \gamma^2 \right| \frac{e^{-\sqrt{\alpha^2 + \beta^2 + \gamma^2}}}{\sqrt{\alpha^2 + \beta^2 + \gamma^2}} d\alpha d\beta d\gamma\\
&= \frac{M^2}{4\pi} \int_{\R}\int_0^{2\pi} \int_0^{\infty} \rho \left| \rho^2 - \frac{M^2}{4}\beta^2 \right| \frac{e^{-\sqrt{\rho^2 + \beta^2}}}{\sqrt{\rho^2 + \beta^2}} d\rho d\phi d\beta\\
&= \frac{M^2}{2} \int_{\R} \int_0^{\infty} \rho \left| \rho^2 - \frac{M^2}{4}\beta^2 \right| \frac{e^{-\sqrt{\rho^2 + \beta^2}}}{\sqrt{\rho^2 + \beta^2}} d\rho d\beta\\
&= \frac{M^2}{2} \int_{-\frac{\pi}{2}}^{\frac{\pi}{2}} \int_0^{\infty} r^3 \cos{\theta} \left| \cos^2{\theta} - \frac{M^2}{4}\sin^2{\theta} \right| e^{-r} dr d\theta\\
&= 3M^2 \int_{-\frac{\pi}{2}}^{\frac{\pi}{2}} \cos{\theta} \left| \cos^2{\theta} - \frac{M^2}{4}\sin^2{\theta} \right| d\theta\\
&= 6M^2 \left( \frac{M^2}{12} - \frac{2}{3} + \frac{4}{3}\sqrt{\frac{4}{4+M^2}} \right)
\end{align}
where we used two changes of variables:
$$\begin{cases}
\alpha = \rho \cos\phi\\
\beta = \beta\\
\gamma = \rho \sin\phi
\end{cases}
\qquad  
\begin{cases}
\rho = r \cos\theta\\
\beta = r \sin\theta
\end{cases}$$
and then solved the integral in the $\theta$ variable finding explicitly the intervals of positivity and negativity of the function in the absolute value.

We have to compute now the density of $\Tilde{\sigma}_{P_\lambda}(W_0)$ at $0$. It is a gaussian random vector with zero mean and covariance
\[
\Sigma=
\begin{bmatrix}
\frac{M^2}{12}+\frac{N^2}{48} & 0 & -\frac{M^2}{12}+\frac{N^2}{48} & 0 \\
0 & \frac{5M^2}{12}+\frac{N^2}{48} & 0 & -\frac{M^2}{12}+\frac{N^2}{48}\\
-\frac{M^2}{12}+\frac{N^2}{48} & 0 &  \frac{5M^2}{12}+\frac{N^2}{48} & 0 \\
0 & -\frac{M^2}{12}+\frac{N^2}{48} & 0 & \frac{M^2}{12}+\frac{N^2}{48}
\end{bmatrix}
\]
that can be computed by looking at the coefficients of $P_\lambda$ in the monomials $x_0^3, x_0^2x_1, x_0x_1^2, x_1^3$.
This implies that
\[
p(0,W_0) = \frac{1}{4\pi^2 \sqrt{\det \Sigma}} = \frac{3}{2\pi^2 \left(4M^2 - \frac{M^4}{2}\right)}
\]
where we simplified the expression of the determinant using the relation $\frac{M^2}{6}+\frac{N^2}{48}=1$.
Finally the volume of the Grassmannian \cite[Remark $2$]{BLLP} is $\vol(Gr(2,4))=2\pi^2$, therefore we have that
\begin{equation}
\check{E}_{M,N} = \frac{12}{8-M^2}\left( \frac{M^2}{4} - 2 + 4\sqrt{\frac{4}{4+M^2}} \right).
\end{equation}
Observe that $M=\mu(\lambda)\lambda=\frac{4\sqrt{3}}{\sqrt{(1-\la)^2 + 8\la^2} }\la$, so we can come back to the original parameter $\la$ and obtain that
\be E_\la=\frac{9(8\la^2+(1-\la)^2)}{2\la^2+(1-\la)^2}\left(\frac{2\la^2}{8\la^2+(1-\la)^2}-\frac{1}{3}+\frac{2}{3}\sqrt{\frac{8\la^2+(1-\la)^2}{20\la^2+(1-\la)^2}}\right).\ee

\end{proof}

\subsection{Properties of the function \texorpdfstring{$E_{\lambda}$}{}}

\begin{proposition}
The function $E_{\la}$ is monotone increasing.
\end{proposition}

\pgfplotsset{every axis x label/.append style = {font = \relsize{-1} },every axis y label/.append style = {font = \relsize{-1}, rotate = -90, yshift = 5em, xshift = 0.5em} } 

\begin{figure}[ht]
\begin{center}
\begin{tikzpicture}
\begin{axis}[
domain=0:1,
xtick = {0,0.33,0.66,1},
xticklabels = {0, 1/3, 2/3, 1},
ytick  = {1,2,3,4,5,6,7,8,9,10,11,12,13},
xmin=0, xmax=1.10,
ymin=0, ymax=13.40,
width = 4cm, height=10cm,
grid,
]
\addplot[color=red,line width=1pt]
{(9*(8*x^2+(1-x)^2)/(2*x^2+(1-x)^2)*(2*x^2/(8*x^2+(1-x)^2)-1/3+2/3*sqrt((8*x^2+(1-x)^2)/(20*x^2+(1-x)^2)))};
\end{axis}
\end{tikzpicture}
\caption{A plot of the function $E_{\la}$.}\label{plot}
\end{center}
\end{figure}
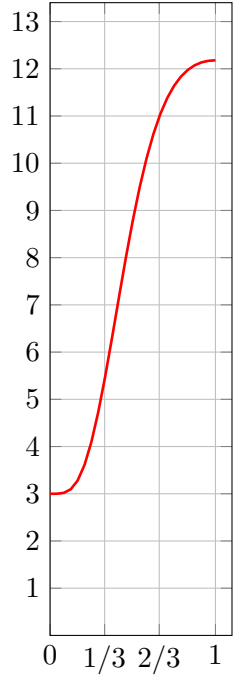

\begin{proof}
In order to simplify the computation and because $M$ is an increasing function of $\la$, we will prove the monotonicity of $E$ as a function of $M$ instead of $\la$.
\[
\check{E}_{M,N}=-3+\frac{96}{(8-M^2)(\sqrt{4+M^2})}.
\]
Then, it is enough to prove that the denominator $g(M)=(8-M^2)(\sqrt{4+M^2})$ is decreasing. In fact,
\begin{align}
 g'(M) &=-2M(\sqrt{4+M^2})+\frac{(8-M^2)M}{\sqrt{4+M^2}}\\
 &=\frac{-2M(4+M^2)+M(8-M^2)}{\sqrt{4+M^2}}\\
 &=\frac{-3M^3}{\sqrt{4+M^2}}.\\
\end{align}
So for positive values of $M$, which are the ones we are interested in, $g'(M)\leq 0$ and therefore $E_{\la}$ is increasing.

\end{proof}

The plot of the function $E_{\la}$ is shown in figure \ref{plot}. Its minimum is $E_{0}=3$ whereas the maximun is reached by the other limit case $\la=1$ and is $E_{1}=24\sqrt{\frac{2}{5}}-3\simeq12,179$, as stated in Corollary \ref{cor_maximum}. This value of $\la$ corresponds to purely harmonic polynomials of degree $3$.

Let us focus now on the minimum. The fact that it is $3$ may be proved also with another approach, that gives some information about the deterministic situation. 
We now define different discriminants and explain the relation between them. 

\begin{table}
\begin{center}
\begin{tabular}{|c|c|c|}
\hline
 & Topology of zero set & Number of lines \\ \hline 
\centering
1 & $\RP^2$\#$3T^2$ & 27  \\ \hline
\centering

2 & $\RP^2$\#$2T^2$ & 15\\ \hline

3 & $\RP^2$\#$T^2$ & 7 \\ \hline

4 & $\RP^2$ & 3\\ \hline

5 & $\RP^2\sqcup S^2$ & 3\\ \hline

\end{tabular}
\end{center}
\caption{We present here in a schematic way the connected components of $\RP^{19}\setminus\Delta^{\C}$ with the topology of the zero sets of the cubics in them and the number of lines on those cubics.}\label{tbl:toplin}
\end{table}

We call \textit{complex discriminant} the subset of those cubics in $\RP^{19}$ which have a complex singularity (their partial derivatives have a common complex zero) and we denote it with $\Delta^{\C}$. It is a known fact (see \cite{ACD}) 
that $\RP^{19}\setminus\Delta^{\C}$ has five connected components. If we fix a connected component among those five, all zero sets of  cubics in there contain the same number of lines and are all homotopy equivalent (see Table \ref{tbl:toplin}). We call \textit{real discriminant} the subset of cubics in $\Delta^{\C}$ such that at least one singularity is real and we will denote it by $\Delta^{\R}$.  Notice that every smooth cubic contains $27$ lines and therefore every element in $\RP^{19}\setminus\Delta^{\C}$ contains a finite number of lines.

In the next Proposition we will work in the space $W_{3,3}$ endowed with the $L^2$ norm. With some abuse of notation we denote by $\Delta^{\C} \subset W_{3,3}$ the set of those $g\in W_{3,3}$ whose projectivization is in $\Delta^{\C}$ or $g$ is the zero cubic. Same for $\Delta^{\R}$.

\begin{proposition}\label{prop_perturbH1}
There exists $\epsilon >0$ such that for all
\begin{equation}
    h_1 (x) = \| x \|^2 ( a_0 x_0 + a_1 x_1 + a_2 x_2 + a_3 x_3 ) \in \|x\|^2\mathcal{H}_1^3
\end{equation}
with $\norm{h_1}^2_{L^2} = 48 (a_0^2 + a_1^2 + a_2^2 + a_3^2) = 1$ and for all $g\in W_{3,3} \setminus \Delta^{\C}$ such that $\| g- h_1 \|_{L^2} \leq \epsilon$ the zero set of $g$ contains exactly $3$ lines.
\end{proposition}

\begin{proof}
First of all notice that if we find the $\epsilon$ of the claim for a fixed $\bar{h}_1$, then the same $\epsilon$ works for any other polynomial $h_1$ as in the statement. In fact $\exists R\in O(4)$ such that $h_1 (x) = \bar{h}_1 (Rx)$, where the $L_2$ norm of $h_1$ is again $1$. Moreover, due to the $O(4)$--invariance of the $L^2$--norm, that rotation $R$ takes the ball of radius $\epsilon$ around $\bar{h}_1$ into the ball of radius $\epsilon$ around $h_1$, without changing the geometry of the zero sets of the cubics in there.
So we are left to prove the claim for a fixed $h_1$.
The cubic $h_1$ belongs to $\Delta^{\C}\setminus\Delta^{\R}$, and its zero set is topologically $\RP^{2}$. Thanks to Thom's isotopy lemma \cite{thomhomologie} $\exists\epsilon>0$ such that for $g\in W_{3,3}$ with $\norm{g-h_1}_{L^2}\leq\epsilon$ the zero sets $Z(g)$ and $Z(h_1)$ are ambient isotopic, and hence homeomorphic. This means that if $g$ is a cubic as above and $g\not\in \Delta^{\C}$, then the projectivization of $g$ belongs to the connected component of $\RP^{19}$ whose zero set is topologically $\RP^2$ and contains exactly $3$ lines.
\end{proof}

\begin{remark}\label{rmk_norm_a}
If $\|h_1\|_{L^2}\neq 1$, the claim above remains true but in a slightly different neighborhood: for all $g\in W_{3,3} \setminus \Delta^{\C}$ such that $\| g- h_1 \|_{L^2} \leq \epsilon\|h_1\|_{L^2}$, the zero set of $g$ contains exactly $3$ lines.
\end{remark}

In view of Proposition \ref{prop_perturbH1}, we can deduce that
\begin{equation}
    \lim_{\lambda \to 0} E_{\lambda} = 3
\end{equation}
without knowing the explicit formula for $E_\lambda$. In fact given any polynomial $f\in W_{3,3}$, thanks to the harmonic decomposition, we can always write $f=h_3+h_1$ where $h_3\in\mathcal{H}_3^3$ and $h_1\in\norm{x}^2\mathcal{H}_1^3$. Taking the $\epsilon$ of the proposition above, we have \begin{equation}
    \begin{split}
        E_{\lambda} &= \mathbb{E}\#\lbrace \text{lines on } Z(f) \rbrace \\
        &= \frac{1}{K} \int_{\|x\|^2\mathcal{H}_1^3} \int_{\mathcal{H}_3^3} \#\lbrace \text{lines on } Z(f) \rbrace\: e^{-\left( \frac{\|h_3\|_{L^2}^2}{2\lambda^2} + \frac{\|h_1\|_{L^2}^2}{2(1-\lambda)^2} \right)} \frac{1}{\lambda^{16}(1-\lambda)^{4}} dh_3 \, dh_1 \\
        &= \frac{1}{K} \int_{S^3}\Theta(\theta)\,\int_0^{\infty} \left[ \int_{ \left\lbrace \|h_3\|_{L^2} \leq \rho\, \epsilon \right\rbrace} \cdots \; \, dh_3 + \int_{ \left\lbrace \|h_3\|_{L^2} > \rho\, \epsilon \right\rbrace} \cdots \; \, dh_3 \right]\cdot \rho^3 \, d\rho \, d\theta\\
    \end{split}
\end{equation}
where $K$ is a normalization constant, $\Theta(\theta)$ a function deriving from the spherical change of coordinates, $S^3 = \{ h_1 \in \|x\|^2\mathcal{H}_1^3 \: \text{such that} \: \norm{h_1}_{L^2}=1 \}$, $\rho =\norm{h_1}_{L^2}$. The number of lines on $Z(f)$ in the first summand is exactly $3$ for the generic $f$ because we are in the nice neighborhood of Proposition \ref{prop_perturbH1}. Therefore $E_{\lambda} = 3\cdot\mathbb{P}(h_3+h_1\in B(h_1,\epsilon\norm{h_1}_{L^2}))+ I(\lambda)$ where $I(\lambda)$ is the following non--negative integral
\begin{equation}
    \begin{split}
        & \quad\frac{1}{K} \int_{S^3}\Theta(\theta)\int_0^{\infty}  \int_{ \left\lbrace \|h_3\|_{L^2} > \rho\, \epsilon \right\rbrace} \#\lbrace \text{lines on } Z(f) \rbrace\: e^{-\left( \frac{\|h_3\|_{L^2}^2}{2\lambda^2} + \frac{\rho^2}{2(1-\lambda)^2} \right)}\frac{ \rho^3}{\lambda^{16}(1-\lambda)^{4}} dh_3\,d\rho \, d\theta \\
        &\leq \frac{27}{K} \int_{S^3} \Theta(\theta) \, d\theta \int_0^{\infty}  \int_{ \left\lbrace \|h_3\|_{L^2} > \rho\, \epsilon \right\rbrace} e^{-\left( \frac{\|h_3\|_{L^2}^2}{2\lambda^2} + \frac{\rho^2}{2(1-\lambda)^2} \right)}\frac{ \rho^3}{\lambda^{16}(1-\lambda)^{4}} dh_3\,d\rho \\
        &\leq K' \int_0^{\infty}  \int_{ \left\lbrace \lambda\|\hat{h}_3\|_{L^2} > \rho\, \epsilon \right\rbrace} e^{-\left( \frac{\|\hat{h}_3\|_{L^2}^2}{2} + \frac{\rho^2}{2(1-\lambda)^2} \right)}\frac{\rho^3}{(1-\lambda)^{4}} d\hat{h}_3\,d\rho \\
    \end{split}
\end{equation}
for $K'$ a new constant that englobes all the others, and $\hat{h}_3=\frac{h_3}{\lambda}$. Using dominated convergence  it is now easy to see that $\lim_{\lambda\to 0} I(\lambda) =0$. Indeed the last inequality of the explicit computation also proves $\mathbb{P}(h_3+h_1\not\in B(h_1,\epsilon\norm{h_1}))\to 0$, which implies $\mathbb{P}(h_3+h_1\in B(h_1,\epsilon\norm{h_1}))\to 1$ and therefore 
\begin{equation}
    \lim_{\lambda\to 0} E_{\lambda} = 3.
\end{equation}

\subsection{Generalization}
More in general let us consider $f\in \R[x_0, \dots , x_n]_{(d)}$ and the associated zero locus $Z(f)\subset \RP^n$. The same polynomial $f$ defines a section $\sigma_f$ of the vector bundle
\[
\begin{tikzcd}
\hbox{Sym}^d(\tau^*_{2,n+1}) \arrow[d, "\pi"]\\
G(2,n+1)
\end{tikzcd} 
\]
such that $\sigma_f(W)=f\vert_W$. The set $\{\sigma_f =0\}$ corresponds to the lines contained in $Z(f)$ and it is generically $0$--dimensional if and only if $d=2n-3$. So it makes sense to ask for the number of lines inside a hypersurface of degree $2n-3$ in $\RP^n$. If for the case of cubics in $\RP^3$ it was at least known that the maximum number of complex lines, 27, could be reached also in the real case, when moving to this more general setting it is not clear whether the generic number of complex lines can be realized or not by real lines.
Following the same procedure explained in Section \ref{harmonics} we may wonder
\begin{center}
\emph{"What is the expected number of real lines inside a random invariant hypersurface of degree $2n-3$ in $\RP^n$?"}
\end{center}
The idea is that the expectation might be maximized again by purely harmonic polynomials of top degree, and so the possible way of constructing hypersurfaces with many lines could be sampling random pure harmonics of degree $2n-3$.

In \cite{BLLP} the authors have proved that
\[
\lim_{n\to \infty} \frac{ \log E_n^{\textrm{Kostlan}} }{\log C_n}=\frac{1}{2}
\]
where $E_n^{\textrm{Kostlan}}$ is the expected number of real lines inside a random invariant hypersurface $Z(f)$ of degree $2n-3$ in $\RP^n$ sampled from the Kostlan distribution, and $C_n$ is the number of complex lines on a generic hypersurface of degree $2n−3$ in $W_{3,3}^n$. This led A. Lerario to a conjecture: sampling random pure harmonics of degree $2n-3$ instead of Kostlan, the intuition is that
\[
\lim_{n\to \infty} \frac{ \log E_n^{\textrm{Harmonic}} }{\log C_n}>\frac{1}{2}
\]
(or maybe in a wonderful universe the limit could also be equal to $1$). 

\begin{remark}
Such results would be relevant also because they may give some information also about the deterministic case.
In dimension $n>3$, it is not even clear if there exist real hypersurfaces containing $C_n$ real lines, but for sure there must exist hypersurfaces with at least $\lceil E_n \rceil$ real lines.
Random results thus give a bound that may not be known yet.
\end{remark}

\bibliographystyle{spmpsci}
\bibliography{literature}
\end{document}